\documentclass{amsart}
\usepackage[foot]{amsaddr}
\newcommand{\lu}[1]{{\color{magenta}#1}}
\usepackage{hyperref}
\usepackage{enumerate}
\usepackage{todonotes}
\usepackage{amsmath, amsthm, amsfonts, mathabx}
\usepackage[capitalise]{cleveref}
\usepackage{nomencl}
\usepackage{bm}
\usepackage{mathrsfs}

\usepackage{amscd}
\usepackage[english]{babel}
\usepackage{xcolor}
\usepackage[capitalise]{cleveref}
\usepackage{verbatim}

\newtheorem{theorem}{Theorem}
\numberwithin{theorem}{section}

\newtheorem{definition}[theorem]{Definition}

\newtheorem{remark}[theorem]{Remark}
\newtheorem{corollary}[theorem]{Corollary}
\newtheorem{lemma}[theorem]{Lemma}
\newtheorem{observation}[theorem]{Observation}
\newtheorem{fact}[theorem]{Fact}

\usepackage{mathabx}

\newcommand{\nin}[0]{\notin}

\newcommand{\enum}[2]{\begin{enumerate}[\hspace*{0.5cm}#1] #2 \end{enumerate}}
\newcommand{\N}[0]{\mathbb{N}}
\newcommand{\G}[0]{\mathbb{G}}

\newcommand{\nil}[0]{\emptyset}
\newcommand{\Q}[0]{\mathbb{Q}}
\newcommand{\ci}[0]{\widebar{int}}

\title{Planar graphs with separation are dp-minimal}
\author{J. de la Nuez Gonz\'alez }
\email{jdelanuez@kias.re.kr}
\address{Korea Institute for Advanced Study (KIAS)}
\date{\today}
\begin{document}

\begin{abstract}
We prove that given a planar embedding of a graph in the sphere the expansion of the graph structure by predicates encoding separation of vertices by simple cycles of the graph is dp-minimal.
\end{abstract}
\maketitle
\section{Introduction}
  
  The notion of an NIP theory is a central to classification theory.
  While it is usually defined in terms of the absence of uniformly definable families of sets shattering infinite families of tuples, an alternative characterization can be given in terms of the boundedness of so called dp-rank (see \cite{simon2015guide} Section 4.2), which describes the behavior of collections of mutually indiscernible sequences under the addition of parameters.
  
  The aim of this note is to provide a family of simple geometrical examples of non-stable theories that are dp-minimal, i.e. for which the dp-rank equals one.
  We will assume the reader is familiar with indiscernible sequences and other basics of model theory. We refer to \cite{tent2012course} for a general introduction to model theory and to \cite{simon2015guide} for an introduction to NIP theories. We work with indiscernible sequences indexed by arbitrary infinite linear orders.

  \begin{definition}
  Two indiscernible sequences $I=(a_{j})_{j\in J}$ and $I'=(a'_{j})_{j\in J'}$ are mutually indiscernible if $I$ is indiscernible over $I'$ and viceversa.
  \end{definition}
  
  \begin{definition}
  A theory $T$ is dp-minimal if for any model of $M$ of $T$, any two mutually indiscernible sequences of tuples $(a_{i}^{1})_{i\in J_{1}}$ and $(a_{i}^{2})_{j\in J_{2}}$ and any element $b$ in $M$ there is $l\in\{1,2\}$ such that
  $(a_{i}^{l})_{i\in J_{l}}$ is indiscernible over $b$.
  \end{definition}

  Fix a countable planar graph $G=(V,E)$ and consider a planar embedding $\iota:|G|\to S^{2}$ of its geometric realization into the sphere. With $\iota$ we associate an expansion $G^{\iota}$ of the graph structure $G$ as follows. For each $n\in\N$ we introduce a new $(n+2)$-ary predicate symbol $R_{n}$ and we let $G^{\iota}\models R_{n}(a,b,c_{1,},c_{2},\cdots c_{n})$  if and only $c_{1},c_{2},\cdots c_{n}$ is the sequence of vertices read along a simple cycle $\gamma$ of the graph, possibly with cyclically contiguous repetitions,  $\{a,b\}\cap\{c_{1},c_{2},\cdots c_{n}\}=\emptyset$ and $\iota(a),\iota(b)$ lie on the same connected component of
  $S^{2}\setminus\iota(|\gamma|)$.

  \newcommand{\ppb}[2]{\mathbf{D}^{+}_{#1}(#2)}
  \newcommand{\nnb}[2]{\mathbf{D}^{-}_{#1}(#2)}
  \newcommand{\is}[2]{(#1_{j})_{j\in#2}}
  \newcommand{\ajs}[0]{(a_{j})_{j\in J}}
  We prove the following:
  \begin{theorem}
  \label{main theorem} The theory $T=Th(G^{\iota})$ is dp-minimal
  \end{theorem}
  
  Note that the theory of a pure planar graph $G$ is known to be stable. In fact \cite{bobkov2017computations} shows that much more in general the theory of any superflat graph is dp-minimal. Superflat graphs are also long known to be stable (see \cite{podewski1978stable}). The structure $G^{\iota}$ in generally is not.
  
  
\section{Basic observations and definitions}
  \newcommand{\side}[0]{side} 
  
  From now on we will work in the monster model $\mathbb{G}^{\iota}$ of $G^{\iota}$, the elements of whose universe we will keep referring to as vertices. We will refer to a tuple $(c_{1},c_{2}\dots c_{n})$, possibly with repetitions, read along a simple path in the graph as an $n$ cycle or simply a cycle. If we drop the condition that $c_{1}$ and $c_{n}$ and connected we obtain what we will call a path (so by a path we will always intend a simple path). By a subpath of a path (cycle) we will intend some subsequence (of some cyclic permutation) of it.
  Given an $n$-cycle $c$ and $a,b$ outside of $c$  we say that $c$ does not separate $a$ and $b$ if $R_{n}(a,b,c)$ holds. We abbreviate this with the expression $a\sim_{c} b$. We also write $\ppb{c}{a}:=\{b\,|\,a\sim_{c}b\}$ and
  $\nnb{c}{a}:=\{b\,|\,a\nsim_{c}b\}$.

  Let us single out some basic properties encoded in $T$ resulting from the definition of $G^{\iota}$:
  \enum{(A)}{
  \item \label{propA} If $a_{1},a_{2}\cdots a_{m}$ is a path and
  $a_{1}\nsim_{c} a_{m}$ for some cycle $c$, then there are $j<j'\in\{1,2,\dots n\}$ such that $j'-j>1$,
  $a_{j},a_{j'}\nin c$, $a_{j}\nsim_{c} a_{j'}$, and $a_{j+1}, a_{j+2},\dots a_{j'-1}$ is a subpath of $c$.
  \item \label{propB} Given two cycles $c_{1},c_{2}$ such that $c_{1}\setminus(c_{1}\cap c_{2})$ lies on a single side of $c_{2}$, then the same is true if we exchange the roles of $c_{1}$ and $c_{2}$. If this is not the case, then we say that $c_{1}$ and $c_{2}$ \emph{cross}.
  
  \item \label{propC}If $c_{1}$ and $c_{2}$ do not cross, then $\sim_{c_{1},c_{2}}$ divides $U\setminus c_{1}\cup c_{2}$ in at most three classes, i.e. there is one side of $\gamma_{1}$ and another of $\gamma_{2}$ which are nested.
  }
  \begin{remark}
  \label{non crossing}By \ref{propA} if the intersection of $c_{1}$ and $c_{2}$ is a single path (in particular, if it contains a single vertex), then $c_{2}$ and $c_{2}$  cannot cross.
  \end{remark}
  \newcommand{\dk}[1]{C_{#1}}
  \newcommand{\cdk}[1]{\bar{C}_{#1}}
  \newcommand{\C}[0]{\mathcal{C}}
  
  \newcommand{\cc}[0]{\mathring{\gamma}}
  \newcommand{\rr}[0]{\mathring{\rho}}
  \newcommand{\nn}[0]{\mathring{\nu}}
  \newcommand{\dd}[0]{\mathring{\delta}}
  \newcommand{\ee}[0]{\mathring{\epsilon}}
  
  \newcommand{\D}[0]{\mathcal{D}}
  \newcommand{\E}[0]{\mathcal{E}}
  
  \subsection*{Cages and nests}
    
    By a \emph{cage} $\mathcal{C}$ we intend a collection $(\gamma_{j})_{j\in J}$, $|J|\geq 3$ of simple paths of bounded length between two distinct fixed vertices $u$ and $v$, indexed by a linearly ordered set $J$ and such that the following properties are satisfied:
    \begin{itemize}
    \item $\gamma_{i}\cap\gamma_{j}=\{u,v\}$ for $i\neq j$ in
    $J$.
    \item  $\gamma_{i}\ast\gamma_{k}^{-1}$  separates $\gamma_{l}\setminus\{u,v\}$ from $\gamma_{j}\setminus\{u,v\}$ if and only if up to exchanging $j$ and $l$ and exchanging $i$ and $k$ either $i<j<k<l$ or $l<i<j<k$ .
    \end{itemize}

    Let $\mathcal{C}=(\gamma_{j})_{j\in J}$ be a cage. We will write $pol(\C)=\{u,v\}$, $\cc_{j}=\gamma_{j}\setminus\{u,v\}$ and from this point on use the convention $\gamma_{i,k}$ to denote the path $\gamma_{i}\ast\gamma_{k}^{-1}$.
    
    We will also write $C_{j,k}=C_{k,j}=\ppb{\gamma_{j,k}}{d}$, where $d$ is some vertex of $\cc_{l}$ for some $l$ between
    $j$ and $k$, or $D^{-}_{\gamma_{j,k}}(d)$ where $d$ is some vertex of $\cc_{l}$ for some $l\nin\{j,k\}$ that is not between $j$ and $k$. It is easy to check that the definition does not depend on the choice of $d$.
    Notice that $C_{j,k}\subseteq C_{j',k'}$ for ordered pairs of indices $(j,k),(j',k')\in J^{2}$ if and only if
    $j'\leq j$ and $k\leq k'$.
    

    By a \emph{nest} $\mathcal{N}$ we mean a sequence $(\nu_{j})_{j\in J}$ of simple cycles of bounded length that are either disjoint
    or intersect pair-wise in a unique vertex $v$ and such that for $i<j<k$ in $J$
    the cycle $\nu_{j}$ separates $\nu_{k}\setminus\nu_{j}$ from $\nu_{i}\setminus\nu_{j}$.

    Let $pol(\mathcal{N})$ be equal to $\emptyset$ in the first case or $\{v\}$ in the second. Write $\nn_{j}=\nu_{j}\setminus pol(\mathcal{N})$.
    Now, for some suitable choice of a side $N_{j}$ of $\nu_{j}$ we have $N_{j}\subset  N_{j'}$ for $j< j'$. In analogy with the previous definition, for $j<j'\in J$
    we write $N_{j,j'}:=N_{j'}\setminus(\nu_{j}\cup N_{j})$, $int(\mathcal{N})=\bigcup_{j<j'}N_{j,j'}$.
    
    The following observation is a consequence of (\ref{propB}):
    \begin{observation}
    \label{order property}
    Given an unordered collection $\{\gamma_{j}\}_{j\in J}$ of size at least $3$ of paths of bounded length between vertices (possibly equal) vertices $v$ and $w$  whose mutual intersection is $\{v,w\}$ there is a linear order on $J$, unique up to inversion and cyclic permutation in case $v\neq v$ making $\{\gamma_{j}\}_{j\in J}$ into a cage in case $v\neq w$ or a nest in case $v=w$.
    \end{observation}
    
    Given a cage or nest $\mathcal{D}$, we will refer to the set $int(\mathcal{D}):=\bigcup_{j\neq k}D_{j,k}$ as the \emph{interior} of $\mathcal{D}$ and we will write $\ci(\mathcal{D}):=int(\mathcal{D})\cup pol(\mathcal{D})$. The complement of $int(\mathcal{D})$ will be called the \emph{exterior} of $\mathcal{D}$ and denoted by $ext(\mathcal{D})$.
    
    \begin{definition}
    We say that a sequence of tuples (for instance, a cage or nest) indexed by a linearly ordered set $J$ is \emph{large} if $J$ is dense with no upper or lower bound.
    \end{definition}
    
    \begin{definition}
    We say that a sequence of elements $(a_{i})_{i\in I}$, indexed by a linear order of size at least $3$
    \emph{fits into} a cage or nest $\mathcal{D}=(\delta_{j})_{j\in J}$ if $\{a_{i}\}$ does not intersect any $\delta_{j}$ and
    \begin{itemize}
    \item for any $i_{0}<i_{1}<i_{2}$ in $I$ there are $j_{0}<j_{1}<j_{2}<j_{3}$ in $J$ such that $C_{j_{l},j_{l+1}}\cap\{a_{i_{0}},a_{i_{1}},a_{i_{2}},a_{i_{3}}\}=\{a_{i_{l}}\}$ for $i=0,1,2$
    \item  for any $j<k$ in $J$ there is $i\in I$ with $a_{i}\in C_{j.k}$.
    \end{itemize}
    \end{definition}
    If no such $\mathcal{D}$ exists we say that $(a_{i})_{i\in I}$ is \emph{free}.
    
    \begin{observation}
    \label{cage trimming}If a large sequence $(a_{i})_{i\in I}$ fits into a cage or nest $\D=(\delta_{l})_{l\in L}$, then there exists $L'\subseteq L$ such that $(a_{i})_{i\in I}$ fits into $\D'=(\delta_{l})_{l\in L'}$ and $\D'$ is large.
    
    If $\mathcal{D}=(\delta_{l})_{l\in L}$ is already large, then for any dense subset $L'\subseteq L$ the sequence $(a_{i})_{i\in I}$ fits into $(\delta_{l})_{l\in L'}$.
    \end{observation}
    \begin{proof}
    
    \end{proof}

    \newcommand{\cage}[0]{cage }   
    \newcommand{\con}[0]{\cup}
    
\section{Properties of large cages}
  
  \begin{lemma}
  \label{pole crossing}Let $\mathcal{D}=\{\delta_{j}\}_{j\in J}$ be a large cage or nest. Let $a,b$ be two elements such that either:
  \begin{itemize}
  	\item  $a\in int(\mathcal{D})$ and $b\in ext(\mathcal{D})$ 
  	\item  $a\in\delta_{i}$ and $b\in\delta_{j}$ for distinct $i,j$  
  \end{itemize}
  Then every path from $a$ to $b$ must pass through $pol(\C)$.
  \end{lemma}
  \begin{proof}
  Let $p$ be a path from $a$ to $b$ and assume without loss of generality that $p$ does not cross $pol(\D)$. Now, there exist an infinity of cycles with pair-wise intersection equal to $pol(\D)$ of the form $\delta_{k,l}$ separating
  $a$ from $b$. Since $p$ must intersect either $\dd_{k}$ or $\dd_{l}$ in each of those cases implies $p$ has to contain infinitely many vertices. This settles the first claim.
  \end{proof}
  
  \begin{corollary}
  \label{cycles and cages} Let $\mathcal{C}=\{\gamma_{j}\}_{j\in J}$ be a large cage and $c$ a cycle. Then one of the following holds:
  \begin{enumerate}
  	\item \label{o1}$c\subseteq ext(\C)$
  	\item \label{o2}$c\subseteq \ci(\C)$
  	\item \label{o3}$c$ is a concatenation $d*e$, where $e$ and $d$ are paths between the two poles of $\C$, with $d\subseteq ext(\C)$ and $e\subseteq\ci(\C)$.
  	If $pol(\C)\nsubseteq c$, then there are $c$ can intersect $\cc_{j}$ for at most one value of $j$.
  \end{enumerate}
  \end{corollary}
  \begin{proof}
  Indeed, if $c$ contains points both in $int(\C)$ and $\ci(\C)^{c}$, then by the previous Lemma it must intersect $pol(\C)$ two times and can only exit $int(\C)$ once.
  \end{proof}
  
  A similar argument shows:
  \begin{lemma}
  \label{cycles and nests}If $\mathcal{N}=(\nu_{j})_{j\in J}$ is a large nest, then for any cycle $c$ either $c\subseteq\ci(\mathcal{N})$ or $c\subseteq ext(\mathcal{N})$. In the former case $c$ can intersect $\nn_{j}$ for at most one value of $j$.
  \end{lemma}
  
  \begin{lemma}
  \label{separation in cages} Suppose that $c$ is a cycle and $(a_{i})_{i\in I}$ a large sequence disjoint from $c$ and fitting into a large cage or nest $\mathcal{D}=(\delta_{l})_{l\in L}$. If we partition $I$ into two sets according to the side of $c$ in which $a_{i}$ lies the resulting partition contains a subinterval $J$ of $I$. Moreover:
  \begin{itemize}
  	\item If $\D$ is a nest and $J$ is bounded on both sides or if $pol(\D)\nsubseteq c$, then $J$ is empty or a singleton.
  	\item If $\mathcal{D}$ is a cage and $c$ is as in case \ref{o3} of Corollary \ref{cycles and cages}, then both $J$ and $I\setminus J$ are unbounded and non-empty intervals.
  \end{itemize}
  %
  \end{lemma}
  \begin{proof}
  By Observation \ref{cage trimming} we may assume $\bigcup\mathcal{D}\setminus pol(D)$ is disjoint from $c$.
  
  Consider first the case in which $\D$ is a cage and $pol(\mathcal{D})\subset c$ and. Adding the one or two subpaths of $c$ that lie in $\ci(\mathcal{D})$ to $\D$ we obtain a new cage by Observation \ref{order property} and the result easily follows.
  
  Take $i_{0}<i_{1}<i_{2}< i_{3}$ in $I$. If $c$ separates $\{a_{i_{0}},a_{i_{2}}\}$ from $\{a_{i_{1}},a_{i_{3}}\}$.
  Then some cycle of the form $\delta_{l}$ (if $\D$ is a nest) or $\delta_{l,l'}$ (if $\D$ is a cage) separates $\{a_{i_{0}},a_{i_{1}}\}$ from $\{a_{i_{2}},a_{i_{3}}\}$ and thus must cross $c$. This impossible unless $\D$ is a cage and $pol(\D)\subseteq c$ in which case it is ruled out by the argument above.
  
  Suppose now that $c$ separates $\{a_{i_{0}},a_{i_{1}}\}$ from $\{a_{i_{2}},a_{i_{3}}\}$.
  
  If $\D$ a cage then there are $l<l'$ in $L$ such that $\delta_{l,l'}$ separates $\{a_{i_{1}},a_{i_{2}}\}$ from $\{a_{i_{0}},a_{i_{3}}\}$, which implies that $c$ and $\delta_{l,l'}$ cross, and thus $pol(\D)\subseteq c$.
  
  If $\mathcal{D}$ is a nest there are distinct $l<l'$ such that
  $\delta_{l}$ separates $a_{i_{0}}$ from $a_{i_{1}},a_{i_{2}},a_{i_{3}}$ and $\delta_{l'}$ separates $a_{i_{0}},a_{i_{1}},a_{i_{2}}$ from $a_{i_{3}}$.
  It follows from property \ref{propC} that $\delta_{l}$ and $\delta_{l'}$ must lie in distinct sides of $c$. This implies that $pol(\D)\subseteq c$ and both $J$ and $I\setminus J$ are unbounded (each of the two sides of $c$ contains one side of either $\delta_{l}$ or $\delta_{l'}$).

  \end{proof}
  
  %
  %

  \begin{lemma}
  \label{cage compatibility}Let each of $\D=(\delta_{j})_{j\in J}$ and $\E=(\epsilon_{k})_{k\in K}$ be either a large cage or a large nest with a pole. Then either $pol(\E)\subseteq\ci(\D)$ or $pol(\E)\subseteq ext(\D)$. Moreover if
  $pol(\D)\neq pol(\E)$ then one of the following mutually exclusive possibilities hold:
  \begin{itemize}
  	\item $int(\D)\cap int(\E)=\emptyset$
  	\item up to exchanging $\D$ and $\E$
  	there exist $j,k\in J$ such that $\ci(\D)\setminus pol(\E)\subset E_{j,k}$.
  \end{itemize}

  \end{lemma}
  
  %
  
  \begin{proof}
  
  If $pol(\E)\cap int(\D)$ and $pol(\E)\cap\ci(\D)^{c}$ are both non-empty, then Lemma \ref{pole crossing} implies the existence of $v\in pol(\E)$ contained in $\ee_{i},\ee_{j}$ for different $i\neq j$: a contradiction, hence the first claim.
  
  For the second claim suppose that $|pol(\mathcal{D})|=max\{|pol(\D)|,|pol(\E)|\}$.
  
  Consider first the case in which $pol(\E)\subseteq ext(\D)$.
  Take $j'<k'<l'$ in $K$. In view of Corollary \ref{cycles and cages} we must have $\epsilon_{j'}\cup\epsilon_{k'}\cup \epsilon_{l'}\subseteq ext(\D)$
  and then either $int(\D)\subseteq E_{j',l'}$ or $E_{j',l'}\subseteq ext(\D)$. If the former does not hold for any $j'<l'$, then
  $int(\D)\cap int(\E)=\emptyset$.
  
  Assume now that
  $pol(\E)\subseteq\ci(\D)$. Choose $j_{0}<k_{0}$ in $J$ such that $pol(\E)\setminus pol(\D)\subseteq D_{j_{0},k_{0}}$. We claim that $int(\E)\subseteq D_{j_{0},k_{0}}$. Indeed, otherwise there are infinite sequences $j'_{1}>j'_{2}\dots$ and $k'_{1}<k'_{2}<\dots $ in $J'$ such that   for every $n$ either $\epsilon_{j_{n}}$ or $\epsilon_{k_{n}}$ is not contained in $D_{j_{0},k_{0}}$. Since $pol(\E)$ intersects $D_{j_{0},k_{0}}$, it follows that $\delta_{j_{0}}\cup\delta_{k_{0}}$ and $\epsilon_{j_{n}}\cup\epsilon_{k_{n}}$ must intersect in at least two vertices, one of which must lie outside of $pol(\E)$.
  This implies that one of $\delta_{j_{0}}$ or $\delta_{k_{0}}$ contains an infinite number of distinct vertices: a contradiction.

  \end{proof}

  \newcommand{\sca}[0]{(\gamma_{j})_{j\in J}}
  \newcommand{\scap}[0]{(\beta'_{j})_{j\in J'}}
  \begin{corollary}
  \label{fitting}Let $\mathcal{D}=(\delta_{j})_{j\in J}$ and $\mathcal{D'}=(\delta'_{j})_{j\in J'}$ be either a large cage or a large nest. If a large sequence $(a_{i})_{i\in I}$ fits into both $\mathcal{D}$ and $\mathcal{D}'$, then $\mathcal{D}$ and $\mathcal{D}'$ are either both cages or both nests, $pol(\mathcal{D})=pol(\mathcal{D}')$ and $int(\mathcal{D})=int(\mathcal{D}')$.
  \end{corollary}
  \begin{proof}
  
  From Lemmas \ref{separation in cages} and \ref{cage compatibility} it follows that $pol(\D)=pol(\D')$.
  
  After replacing $J$ and $J'$ with cofinal dense subsets we may
  assume $\dd_{j}$ and $\dd'_{j'}$ are disjoint for any $j\in J$ and $j'\in J'$.
  Using Lemma \ref{separation in cages} in case both $\D$ and $\D'$ are nests we  conclude that we can merge $\D$ and $\D'$ into a larger cage or nest $\mathcal{D}''=(\delta_{j})_{j\in J''}$, $J'=J'\cup J'$. The assumption that $(a_{i})_{i\in I}$ fits into both $\mathcal{D}$ and $\mathcal{D}'$ implies $J$ and $J'$ are both dense and cofinal in $J''$, from which it follows that $int(\D)=int(\D')$.
  \end{proof}

  As a further corollary we get:
  \begin{corollary}
  \label{invariance of cage interior}Let a sequence of elements $(a_{i})_{i\in I}$ fit into a large cage or nest $\D$. Then $int(\mathcal{D})$ and $pol(\mathcal{D})$  are invariant under any automorphism of $\mathbb{G}^{\iota}$ fixing $\{a_{i}\}_{i\in I}$ set-wise.
  \end{corollary}

\section{Indiscernibles and cages}
  
  
  We start with a couple of observations regarding cages and indiscernible sequences of elements.
  
  \begin{lemma}
  \label{parameters and cages}Let $(a_{j})_{j\in J}$ be a large sequence of elements indiscernible over some set $B$ of parameters and assume that
  $(a_{j})_{j\in J}$ fits into a large nest or cage $\mathcal{D}$. Then $B\subseteq ext(\mathcal{D})$.
  \end{lemma}
  \begin{proof}
  We will focus on the case in which $\mathcal{D}=(\delta_{l})_{l\in L}$ is a cage, since the case in which it is a nest is formally identical.
  Assume the result is false and pick $l_{2}<l_{3}\in L$ such that $b\in C_{l_{2},l_{3}}$.
  Now, choose $l_{0},l_{1},l_{4},l_{5}$ such that $l_{0}<l_{1}<l_{2}$ and $l_{3}<l_{4}<l_{5}$. By definition, there exists $j_{0}<j_{1}<j_{3}<j_{4}$ in $J$ such that $a_{j_{m}}\in C_{l_{m},l_{m+1}}$, $m\in\{0,1,3,4\}$.
  Let $N$ be the maximum length of the cycles $\gamma_{l,l'}$ and consider the following formula over $b$:
  \begin{align*}
  \phi(x_{1},x_{2},x_{3})\cong\exists z_{1}\dots z_{N}\,R_{N}(x_{1},b,z)\wedge \neg R_{N}(x_{2},b,z)\wedge \neg R_{N}(x_{3},b,z)
  \end{align*}
  \end{proof}
  
  The cycle $\gamma_{l_{1},l_{3}}$ witnesses $\models\phi(a_{j_{1}},a_{j_{3}},a_{j_{4}})$.
  On the other, after removing finitely many components from $(a_{j})_{j\in J}$ and $\D$ we may assume one can add $b$ to the sequence $(a_{j})_{j\in J}$ between the position of $a_{j_{1}}$ and $a_{j_{3}}$ and obtain a new sequence fitting into $\D$. Lemma \ref{separation in cages} then implies the formula $\phi(a_{j_{0}},a_{j_{1}},a_{j_{3}})$ cannot hold, contradicting indiscernibility of $(a_{j})_{j\in J}$ over $B$.
  \begin{observation}
  \label{subcages}Assume that $(a_{j})_{j\in J}$ is an indiscernible sequence and there is some infinite $J'\subseteq J$ such that $(a_{j})_{j\in J'}$ fits into a cage (nest). Then $(a_{j})_{j\in J}$ fits into a cage (nest).
  \end{observation}
  \begin{proof}
  By compactness, if any finite subsequence of a sequence $I=(a_{j})_{j\in J}$ fits into a cage with components of uniformly bounded length, then so does $I$.
  \end{proof}

  \begin{lemma}
  \label{indiscernible cycles} Suppose we are given an indiscernible sequence of cycles, $(c_{j})_{j\in J}$.
  Then one of the following holds:
  \enum{(a)}{
  \item \label{case1} Either the sequence $(c_{j})_{j\in J}$ is a nest $\mathcal{D}^{0}$ or it contains finitely many cages $\mathcal{D}^{1},\dots\D^{m}$  with disjoint interiors.
  
  In this case given vertices $b$ and $b'$ that do not appear in the $c_{j}$ we have the following:
  \begin{itemize}
  \item If $(c_{j})_{j\in J}$ is indiscernible over $b$ there is one alternation in the truth value of $b\sim_{c_{j}}b'$ as $j$ ranges in $J$ if $b\in int(\D^{l})$ for $\D^{l}$ and none otherwise
  \item If $b',b''$ belong to $int(\mathcal{D}^{l})$, where $\mathcal{D}^{l}=(\delta^{l})_{j\in J}$ is as above, then there are   two alternations in the truth value of $b'\sim_{c_{j}}b''$ only if
  $j<j'$ in $J$ with $\{b',b''\}\cap D_{j,j'}^{l}=\{b'\}$ exists.
  \end{itemize}
  .

  \item \label{case2} We have $|c_{j}\cap c_{j'}|=2$ for distinct $j,j'$ and if $(c_{j})_{j\in J}$ is indiscernible over $b$, then $\{\nnb{c_{j}}{b}\}_{j\in J}$ is a family of disjoint sets.
  }
  \end{lemma}
  \begin{proof}
  We will consider the case in which $c_{j}$ and $c_{j'}$ have at least two vertices in common for $j,j'\in J$. If not, then $(c_{j})_{j\in J}$ is easily seen to be a nest and the required property can be checked using simpler versions of the same arguments.
  
  We start by noticing that given a linearly ordered set $J$ any indiscernible sequence $(\gamma_{j})_{j\in J}$ of paths between two vertices $u,v$ intersecting pair-wise only on $\{u,v\}$  which is indiscernible must be a cage with respect to the same linear order on $J$.
  
  Suppse $c_{i}=(c_{i}^{k})_{k=1}^{n}$. If we let $F\subset\{1,\cdots n\}$ be the set of indices such that $c_{j}^{l}$ is constant for all $j$ then $c_{j}^{k}\neq c_{j'}^{k'}$ whenever $j\neq j'$ and at least one of $k,k'$ does not belong to $F$.
  Let $\mathcal{P}$ be the collection of all pairs $(k,l)$ of cyclically consecutive indices of $F$ (in that order) which are not cyclically consecutive in $\{1,2\cdots n\}$ and for any $(k,l)\in\mathcal{P}$ let $\gamma^{k,l}_{j}$ be the subpath $c_{j}$ from $c_{j}^{k}$ to $c_{j}^{l}$. Then $\mathcal{C}^{k,l}:=(\gamma_{j}^{k,l})_{j\in J}$ is a cage.
  
  \begin{observation}
  For any distinct $(k_{0},k_{1}),(k_{2},k_{3})\in\mathcal{P}$ and any $j_{1}<j_{2}$ in $J$ it cannot be that $\cc^{k_{2},k_{3}}_{j}$ contains vertices in $int((\gamma_{j})_{j\in(j_{1},j_{2})})$ for $j\leq j_{1}$ or $j\geq j_{2}$.
  \end{observation}
  \begin{proof}
  Otherwise it follows from Lemma \ref{parameters and cages} that the sequence
  $(\cc_{j}^{k_{0},k_{1}})_{j\in (j_{1},j_{2})}$ is not indiscernible over $c_{j_{1}}$ or $c_{j_{2}}$ respectively, contradicting the indiscernibility of $(c_{j})_{j\in J}$.
  \end{proof}
  
  Assume first that $\mathcal{P}$ is not of the form $\{(k,l),(l,k)\}$ and pick
  any distinct $(k_{0},k_{1}),(k_{2},k_{3})\in\mathcal{P}$. We claim that $int(\C^{(k_{0},k_{1})})\cap int(\C^{(k_{2},k_{3})})=\emptyset$. Otherwise by Lemma \ref{cage compatibility} there exist $j_{1}<j_{2}$ in $J$ such that $int(\C^{k_{0},k_{1}})\subseteq C^{k_{2},k_{3}}_{j_{1},j_{2}}$, contradicting the observation above.

  If $\mathcal{P}=\{(k,l),(l,k)\}$ and $int(\C^{l,k})\cap int(\C^{k,l})\neq\emptyset$, then  $int(\C^{k,l})$ contains $\cc_{j}^{l,k}$ for some $j_{0}\in J$. By the observation above no
  $v\in\cc_{j}^{l,k}$ can belong to $C^{k,l}_{j',j''}$ for $j'<j''<j_{0}$ or $j_{0}<j'<j''$. By indiscernibility $\cc^{l,k}_{j}$ must lie in $C^{k,l}_{j',j''}$ for any $j'<j<j''$.
  
  So we have the following alternative:
  \begin{enumerate}
  \item \label{a} $int(\mathcal{C}^{k,l})\cap int(\mathcal{C}^{k',l'})=\emptyset$ for any two distinct $(k,l),(k',l')\in\mathcal{P}$
  \item \label{b} up to exchanging $\C$ and $\C'$  there is a cage $\mathcal{C}'=(\beta'_{j})_{j\in J\times\{1,2\}}$ where $J\times\{1,2\}$ is ordered lexicographically and $\beta'_{(j,1)}=\beta^{k,l}_{j}$ and $\beta'_{(j,2)}=\beta^{l,k}_{j}$.
  \end{enumerate}
  
  One can now check by inspection that given $b,b'$ outside of $\cup_{j\in J}c_{j}$ there is at most one alternation in the truth value of $b\sim_{c_{j}}b'$ for each
  pair $(d,(k,l))$, where $d\in\{b,b'\}$ and $(k,l)\in\mathcal{P}$ and $d\in int(\C^{k.l})$.
  
  In cases \ref{a} there can be at most two such pairs and in fact at most one in case $(c_{j})_{j}$ is indiscernible over $b$, since by Lemma \ref{parameters and cages} the latter implies $b\nin int(\mathcal{C}^{k,l})$ for any $(k,l)\in\mathcal{P}$. If $b,b'\in int(\C^{k,l})$ and there are two alternations, it can be easily seen by inspection that $j,j'$ with the desired properties exist.
  
  In case \ref{b} we have that $\{\nnb{c_{j}}{b}\}_{j\in J}$ is a family of disjoint sets as in the second alternative of the statement.
  \end{proof}

  \begin{lemma}
  \label{alternations}Let $I=(a_{j})_{j\in J}$ be a non-constant large sequence of elements indiscernible over some set $B$ of parameters and $c$ a cycle in which none of the $a_{j}$ appear. If we partition $J$ into two sets $J_{1}$ and $J_{2}$ according to the side of $c$ each $a_{j}$ lies in then either:
  
  \enum{(a)}{
  \item \label{isolated case} up to exchanging $J_{1}$ and $J_{2}$ we have $J_{1}=\{i_{0}\}$, $J_{2}=J\setminus\{i_{0}\}$ and all $b\in B$ lie on the opposite side of $c$ from $a_{i_{0}}$.
  \item $J_{1}$ and $J_{2}$ are both infinite, $J_{1}$ is an interval and $(a_{j})_{j\in J}$ fits into a cage or nest
  }
  \end{lemma}
  \begin{proof}
  Suppose first $B\neq\emptyset$ and fix some $b\in B$. Let $J'=\{j\,|\,a_{j}\sim_{c}b\}$, $J''=J\setminus J'$.
  We claim that it is not possible to have $J_{1}<j_{2}<J_{3}$, where $J_{i}\subseteq J''$ are infinite and $j_{2}\in J'$.
  Indeed, if this is the case, up to extending and then restricting $I$ we may as well assume that $J=J_{1}\cup\{j_{2}\}\cup J_{3}$ and there are $c_{j}$, $j\in J$ such that $(c_{j}a_{j})_{j\in J}$ indiscernible over $b$ and with $c=c_{j_{2}}$. (For $j\in J$ let $c'_{j}$ be the image of $c$ by an automorphism of $\mathbb{G}^{\iota}$ restricting to an order preserving permutation of $(a_{j})_{j\in J}$ sending $j_{2}$ to $j$. Pick an indiscernible sequence indexed over $J$ realizing the EM-type of $(a_{j}c'_{j})_{j\in J}$. Use an automorphism to bring it first components to $(a_{j})_{j\in J}$.) Notice that we then have a double alternation:
  $$a_{j_{2}}\nsim_{c_{j_{1}}}b,\,\,\,a_{j_{2}}\sim_{c_{j_{2}}}b,\,\,\,a_{j_{2}}\nsim_{c_{j_{3}}}b$$
  for $j_{1}\in J_{1}$, $j_{3}\in J_{3}$. According to Lemma \ref{indiscernible cycles} if this happens then $\{\nnb{c_{j}}{b}\}_{j\in J}$ are disjoint, contradicting $\nnb{c_{j}}{b}=\{a_{i}\}_{i\neq j}$.

  Next we show that (independently of $B$) if there are infinite $J_{1}<J_{2}$ in $J$ such that $c$ separates $\{a_{j}\}_{j\in J_{1}}$ from
  $\{a_{j}\}_{j\in J_{2}}$, then $I$ fits into a cage.
  Using Observation \ref{subcages} we may assume that $J=\Q$, $J_{1}=\Q_{<0}$ and $J_{2}=\Q_{\geq 0}$ and that there is an indiscernible sequence $(a_{j}c_{j})_{j\in J}$ with $c_{0}=c$. Lemma \ref{indiscernible cycles} then yields the desired result: every element in the sequence must lie in the interior of either the nest or one of the finitely many cages and this cage should be the same for some infinite subsequence of them. Since for any distinct $j_{0},j_{1}\in J$ there are two alternations in the truth value of $a_{j_{0}}\sim_{c_{j}}a_{j_{1}}$ as $j$ ranges in $J$ it follows that this infinite subsequence fits into a cage and thus the entire sequence does by Observation \ref{subcages}.

  Assume now that $I$ does not fit into a cage or nest and let $\{J,J'\}$ be the partition of $J$ induced by $c$. We argue that there cannot be three alternations between $J$ and $J'$. If this is the case then up to reversing the orientation of $J$
  we have $j_{1}<J_{2}<j_{3}<J_{4}'$, where $j_{1},j_{3}\in J$ and $J_{2},Jl4 \subseteq J'$ are infinite. Since $I_{\restriction J_{2}\cup\{j_{3}\}\cup J_{4}}$ is indiscernible over $a_{j_{1}}$ we obtain a contradiction with the claim of the previous paragraph.
  
  Since any finite sub-interval of $J$ is a singleton, in fact either $J_{1}$ is a singleton or empty. In both cases our first claim ensures that if $I$ is indiscernible over $b$, then $c$ separates $b$ from $I$.
  %
  %
  %

  \end{proof}
  %
  \newcommand{\piv}[0]{pivot}  
  \newcommand{\nod}[0]{nod}

\section{The envelope of a free indiscernible sequence}
  \renewcommand{\S}[0]{\mathcal{S}}
  \newcommand{\X}[0]{\mathcal{X}}
  \newcommand{\Y}[0]{\mathcal{Y}}
  \renewcommand{\H}[0]{\mathcal{H}}
  
  \begin{definition}
  \label{Gj}Given some free non-constant sequence of elements $I=(a_{j})_{j\in J}$
  we let $\S(I)_{j}$ be the collection of all cycles which separate $a_{j}$ from $I\setminus\{a_{j}\}$
  and
  $$\X(I)_{j}=\bigcup_{c\in\S(I)_{j}}\ppb{c}{a_{j}}\cup c.$$
  \end{definition}
  
  For the remainder of the section $I=(a_{j})_{j\in J}$ will be a free non-constant indiscernible sequence of elements.
  
  \begin{definition}
  We say that a vertex in $\X(I)_{j}$ is \emph{external} if it is not contained in $\ppb{c}{a_{j}}$ for any $c\in\S(I)_{j}$.
  \end{definition}

  \begin{definition}
  Let $A$ be a finite collection of vertices. By the complementary regions of $A$ we mean the equivalence classes of the vertices in $V\setminus A$ by the equivalence relation given by:
  \begin{itemize}
  \item  $b\sim b'$ if and only if $b\sim_{d}b'$ for all cycles $d\subseteq A$
  \end{itemize}
  \end{definition}
  
  \begin{lemma}
  \label{non touching} Let $c,c'$ two cycle, $a$ a vertex and suppose that $|c\cap c'|\geq 2$. Then each complementary region $D$
  of $c\cup c'$ is of the form $\ppb{d}{b}$ for some cycle $d\subset c\cup c'$ and some $b\in D$.
  \end{lemma}
  \begin{proof}
  Fix $b\in D$ and consider the collection $\mathcal{P}=\{p_{1},\dots p_{k}\}$ of maximal subpaths of $c'$ with interior in
  $\ppb{c}{b}$. Clearly, for no such subpath $p$ can the two endvertices coincide, since that would imply $p$ is the whole $c'$ and $|c\cap c'|=1$. By induction it is easy to construct cycles $d_{0}=c$, $d_{1},d_{2}\dots d_{k}$ such that the complementary region of $c\cup (c'\setminus\ppb{c}{b})\cup p_{1},\dots p_{k}$ coincides with $\ppb{d_{k}}{b}$.
  \end{proof}
  
  \begin{lemma}
  \label{external are connected}Assume $\S(I)_{j_{0}}\neq\emptyset$ for $j_{0}\in J$. Then any two external vertices of $\X(I)_{j_{0}}$ are contained in a common cycle $c\in\S(I)_{j}$.
  \end{lemma}
  \begin{proof}
  Let $v_{1}$ and $v_{2}$ be external vertices contained in $d_{1},d_{2}\in\S(I)_{j_{0}}$ respectively.
  We may assume $v_{i}\notin d_{3-i}$. It cannot be the case that
  $|d_{1}\cap d_{2}|\leq 1$, since this implies that  $d_{i}\setminus d_{3-i}\subseteq\ppb{d_{3-i}}{a_{j_{0}}}$ for some $i$, so that $v_{i}$ cannot be external.
  
  Thus, we can apply Lemma \ref{non touching} to $d_{1}$ and $d_{2}$. One of the four complementary regions of $d_{1}\cup d_{2}$ contains $a_{j_{0}}$ and by Lemma \ref{alternations} one of the others must contain $\{a_{j}\}_{j\neq j_{0}}$. By Lemma \ref{non touching} the latter must be of the form $\nnb{e}{a_{j_{0}}}$ for some cycle $e\subset d_{1}\cup d_{2}$.
  Since $d_{1}\cup d_{2}\subseteq\ppb{e}{a_{j_{0}}}\cup e$ and $v_{1},v_{2}$ are external, it follows that $v_{1},v_{2}\in e$, as desired.
  \end{proof}
  
  For future reference we also record the following:
  \begin{lemma}
  \label{merging cycles} For any finite set $F\subseteq\X(I)_{j}$ there exists some cycle $c\in\S(I)_{j}$ such that $F\subseteq\ppb{c}{a_{j}}\cup c$.
  \end{lemma}
  \begin{proof}
  Indeed, for any two cycles $c,c'\in\S(I)$ either $|c\cap c'|=1$ and we may assume that $\ppb{c}{a_{j}}\cup c\subseteq\ppb{c'}{a_{j}}\cup c'$ or $|c\cap c'|$ and Lemma \ref{non touching} provides some cycle $e\subseteq c\cup c'$ such that $\ppb{c}{a_{j}}\cup c\cup\ppb{c'}{a_{j}}\subseteq e\cup\ppb{e}{a_{j}}$. The result now easily follows by induction.
  \end{proof}
  
  \begin{lemma}
  \label{one point intersection}  For distinct $j_{1},j_{2}$ in $J$ the intersection $\X(I)_{j_{1}}\cap\X(I)_{j_{2}}$ consists of at most one vertex, external in both $\X(I)_{j_{1}}$ and $\X(I)_{j_{2}}$.
  \end{lemma}
  \begin{proof}
  We may assume $\S(I)_{j}\neq\emptyset$ for all $j$. Suppose $\X(I)_{j_{1}}\cap\X(I)_{j_{2}}$ contains distinct vertices $v,v'$ for $j_{0}\neq j_{1}$.
  
  Notice that given  $c_{1}\in\S(I)_{j_{1}}$ and $c_{2}\in\S(I)_{j_{2}}$ it cannot be that $c_{2}$ lies entirely in $c_{1}\cup\ppb{c_{1}}{a_{j_{1}}}$, or else we would have $\{a_{j}\}_{j\in J\setminus\{a_{j_{0}},a_{j_{1}}\}}\subseteq \nnb{c_{2}}{a_{j_{2}}}\subseteq\ppb{c_{1}}{a_{j_{1}}}$.

  Thus, if the intersection
  $\ppb{c_{1}}{a_{j_{i}}}\cap(\ppb{c_{2}}{a_{j_{3-i}}}\cup c_{3-i})$ is non-empty for some $c_{1}\in\S(I)_{j_{i}}$ and some $c_{2}\in\S(I)_{j_{2}}$, then $c_{1}$ and $c_{2}$ must cross.
  
  Otherwise for any such $c_{1},c_{2}$ we can choose $v$ and $v'$ to be external in both $\X(I)_{j_{1}}$ and $\X(I)_{j_{2}}$.
  
  Remark \ref{non crossing} in the first situation and Lemma \ref{external are connected} in the second there must be $c_{1}\in\S(I)_{j_{1}}$ and $c_{2}\in\S(I)_{j_{2}}$ with at least two vertices in common.
  
  Notice that $a_{j_{1}}$ and $a_{j_{2}}$ lie in different complementary regions from each other and any $a_{j}$, $j\in J\setminus\{j_{1},j_{2}\}$.
  Lemma \ref{non touching} yields some cycle $d\subseteq c_{1}\cup c_{2}$ such that $\nnb{d}{a_{j_{1}}}\cap\{a_{j}\}_{j\in J}=\{a_{j}\}_{j\in J'}$ for some infinite subset $J'\subseteq J\setminus\{j_{1},j_{2}\}$, contradicting Lemma \ref{alternations}.

  \end{proof}
  
  \begin{lemma}
  For distinct $j_{1},j_{2}$ in $J$ there exist a unique vertex  $p_{j_{1},j_{2}}\in\X(I)_{j_{1}}$   such that any path (possibly consisting of a single vertex) from $\X(I)_{j_{1}}$ and $\X(I)_{j_{2}}$ intersecting $\X(I)_{j_{1}}$ only at its endpoints must start at $p_{j_{1},j_{2}}$.
  \end{lemma}
  \begin{proof}
  
  If $\S(I)_{j}=\emptyset$, this is trivial.
  Suppose the condition does not hold and $p_{1}$ and $p_{2}$ minimal paths form $\X(I)_{j_{1}}$ to $\X(I)_{j_{2}}$ (i.e. no subpath has the same property) starting at distinct external vertices, $v_{1}$ and $v_{2}$ respectively and ending at external vertices $w_{1},w_{2}$ of $\X(I)_{j_{2}}$.
  
  Using Lemma \ref{external are connected} in case $p_{1}$ and $p_{2}$ do not intersect, we can find some simple path $q$ between $v_{1}$ and $v_{2}$ whose interior lies outside of $\X(I)_{j_{1}}$.

  By Lemma \ref{external are connected}, there must be two non-degenerate paths $r,r'$ contained in $\X(I)_{j_{1}}$ between $v_{1}$ and $v_{2}$ with $r*r'\in\S(I)_{j_{1}}$. Without loss of generality $d_{1}:=r*q$ and $d_{2}:=r'*q^{-1}$ are cycles.
  
  So for any $i=1,2$ we have that $\{j\in J\,|a_{j}\in\ppb{d_{i}}{a_{j_{1}}}\}$ is equal to either the whole $J$ or  $\{a_{j}\}_{j\in J\setminus\{j_{3}\}}$ for some $j_{3}\in J\setminus\{j_{1}\}$. The latter cannot hold, since it implies that $\{j\,| \,a_{j}\in\ppb{d_{3-i}}{a_{j_{1}}}\}=J\setminus\{j_{1}\}$, i.e, $d_{3-i}\in\S(I)_{j_{1}}$. The former implies that $d_{2}$ separates $\{a_{j}\}_{j\in J}$ into two sets of cardinality at least two, contradicting Lemma \ref{alternations}.
  \end{proof}

  \begin{lemma}
  \label{grapes} Assume that for any distinct $j,j'\in J$ there is a path between the sets $\X(I)_{j}$ and $\X(I)_{j'}$. Then there is a (unique) vertex $\nod(I)$ such that for any distinct $j,j'\in J$
  any path from $\X(I)_{j}$ to $\X(I)_{j'}$ passes through $v$.
  \end{lemma}
  \begin{proof}
  
  For distinct $j,j'\in J$ let $B_{j,j'}$ be the union of $\X(I)_{j}\cup\X(I)_{j'}$ and   the collection $\mathcal{Q}_{j,j'}$ of all the paths between $\X(I)_{j}$ and $\X(I)_{j'}$ intersecting $\X(I)_{j}\cup\X(I)_{j'}$ only at the endpoints. Notice that any automorphism of $\mathbb{G}$ fixing $a_{j}$ and $a_{j'}$ must preserve $\X(I)_{j}$ and $\X(I)_{j'}$ set-wise and thus $B_{j,j'}$.
  
  We may assume $\N\subseteq J$. For $j>2$ let $V_{j}$ the collection of vertices in $B_{1,2}$ that can be joined to $\X(I)_{j}$ by a path whose interior lies outside of $B_{1,2}$.
  
  We claim that $|V_{j}|=1$. Indeed, suppose $V_{i}$ contains distinct vertices $v_{1},v_{2}$. By Lemma \ref{one point intersection} we may assume $v_{1}\nin\X(I)_{1}$. Since any two external vertices of $\X(I)_{j}$ are connected by Lemma \ref{external are connected}, we easily find a path in $\mathcal{Q}_{1,2}$ that is closer to $\X(I)_{j}$ than $v_{1}$ along a path between $v_{1}$ and $\X(I)_{j}$, contradicting $v_{1}\in V_{j}$. Write $\{u_{j}\}=V_{j}\cap B_{1,2}$.
  
  We must have $u_{j}\in\mathcal{Q}_{1,2}$ since otherwise
  $d(\X(I)_{j},\X(I)_{1})\neq d(\X(I)_{j},\X(I)_{2})$, contrary to indiscernibility. Finally, we claim that $u_{j}$ is a constant $u$ for $j>2$.
  Suppose not. Since $u_{j}$ is easily seen to be definable from $a_{1},a_{2},a_{j}$, in that case all $u_{j}$ must be different for $j>2$.
  
  Choose some $q_{j}\in\mathcal{Q}_{1,2}$ with $u_{j}\in q_{j}$.We may assume $(q_{j}u_{j})_{j>2}$ is indiscernible. Using the fact that all $q_{j}$, $j>2$ have the same endpoints one can see from the proof of Lemma \ref{indiscernible cycles} that $u_{j}\in\cc_{j}$, where $\C=(\gamma_{j})_{j>2}$ with $\gamma_{j}$ is a subpath of $q_{j}$ in constant position.
  
  Let $d_{j}\in\S(I)_{j}$ be either $a_{j}$ in case $\S(I)_{j}=\emptyset$ or a cycle in $\S(I)_{j}$ at minimal distance from $A_{1,2}$ in case $\S(I)_{j}\neq\emptyset$. We may assume  $(d_{j}q_{j})_{j\leq 3}$ all have the same type.
  
  If $d_{j}\subseteq ext(\C)$, then by Lemma \ref{pole crossing} any path from $d_{j}$ to any point in $int(\C)$ has to go through $pol(\C)\subseteq B_{1,2}$, contradicting that $V_{j}=\{u_{j}\}$.
  By Lemma \ref{one point intersection} we must have $|d_{j}\cap pol(\C)|\leq 1$ and thus $d_{j}\subseteq int(\C)$, which implies $a_{j}\in int(\C)$ for $j\geq 3$. Since $a_{1},a_{2}\in ext(\C)$, this contradicts Lemma \ref{alternations}.

  \newcommand{\xx}[0]{\X(I)}
  We claim that for any $2<j_{1}<j_{2}$ any path from $X(I)_{j_{1}}$ to $X(I)_{j_{2}}$ must pass through $u$. We can go from $a_{j_{1}}$ to $a_{1}$ by a first using a shortest path $r_{1}$ to $u$ from $a_{j_{1}}$ and then moving inside a path $\mathcal{Q}_{1,2}$ to $a_{1}$. The result is in $\mathcal{Q}_{j_{1},1}$.
  
  Let $r_{2}$ a path from $a_{j_{2}}$ to $u$ conjugate to $r_{1}$ over $a_{1},a_{2}$. If a path from $a_{j_{2}}$ to $a_{j_{1}}$ that does not go through $u$ exists, then either $r_{2}$ does not have minimal length among all paths from $a_{j_{2}}$ to $B_{1,j_{1}}$ or there are two points in $B_{1,j_{1}}$ that are the starting points of paths to $a_{j_{2}}$ with interior outside of $B_{1,j_{1}}$, which contradicts the fact that any automorphism of $\mathbb{G}^{\iota}$ fixing $a_{j_{1}},a_{j_{2}}$ and taking $a_{2}$ to $a_{j_{1}}$ must take $B_{1,2}$ to $B_{1,j_{1}}$.
  
  So for any distinct $j_{1},j_{2},j_{3}$ in $J$ there exist a vertex that must lie in any path between $\X(I)_{j_{k}}$ and $\X(I)_{j_{k'}}$ for any distinct $k,k'\in\{1,2,3\}$. Some simple considerations regarding the distance of such vertex from $\X(I)_{j_{k}}$ yield that in fact said vertex must be the same for any such triple.
  \end{proof}
  
  \begin{definition}
  We let $nod(I)$ equal to the vertex $u$ provided by the previous lemma in case $\X(I)_{j}$ and $\X(I)_{j'}$ are connected by some path for distinct $j,j'$ and $\emptyset$ otherwise.
  \end{definition}
  
  \newcommand{\T}[0]{\mathcal{T}}

  \begin{definition}
  \label{free envelope}Suppose we are given a free indiscernible sequence of elements $I=(a_{j})_{j\in J}$. Let $\H(I)_{j}$ the collection of all vertices connected to $X(I)_{j}$ by a path that does not pass through $nod(I)$.
  
  We define $env_{I}(a_{j})$ as the union of all vertices that are either in $\H(I)_{j}$ or in $\nnb{c}{a_{j'}}$ for some cycle contained in $\H(I)_{j}\cup nod(I)$ and some (equivalently, any) $j'\in J\setminus\{j\}$. Finally, we let $env(I):=\bigcup_{j\in J}env_{I}(a_{j})$.
  \end{definition}

  \begin{observation}
  \label{cycles}Suppose that $c$, $c'$ and $d$ are cycles and $a,a',b,v$ vertices such that $a,a'\nin c\cup c'$ and:
  \begin{enumerate}
  \item  \label{aa} $(\ppb{c}{a}\cup c)\cap(\ppb{c'}{a'}\cup c')\subseteq\{v\}$
  \item \label{cc} $d$ is connected to $c'$ by a path $p$ that does not go through $v$
  \item \label{dd} $c\cap d\subseteq \{v\}$
  \end{enumerate}
  Then $\ppb{c}{a}\cup c\subseteq\ppb{d}{b}\cup d$ or $(\ppb{c}{a}\cup c)\cap(\ppb{d}{b}\cup d)\subseteq\{v\}$.
  \end{observation}
  \begin{proof}
  Suppose the conclusion does not hold. By \ref{aa} and \ref{dd} we have that $c'\setminus c$ and $d\setminus c$ lie entirely on one side of $c$ and in fact the same side, by \ref{cc}. By \ref{aa} this side is $\nnb{c}{a}$. It then follows that $\ppb{c}{a}\cup c\subseteq\ppb{d}{b}\cup d$ or $(\ppb{c}{a}\cup c)\cap(\ppb{d}{b}\cup d)\subseteq\{v\}$
  \end{proof}
  
  We note the following easy fact:
  \begin{observation}
  \label{not through node}If $a\in env_{I}(a_{j})$ is connected with $b$ by a path that does not go through $nod(I)$, then $b\in env_{I}(a_{j})$.
  \end{observation}

  \begin{lemma}
  \label{intersection at the node}The sets $env_{I}(a_{j})$ and $env_{I}(a_{j'})$
  are not connected by a path that does not pass through $nod(I)$.
  \end{lemma}
  \begin{proof}
  Otherwise there are distinct indices $j_{1}$ and $j_{2}$ so that $b\in env_{I}(a_{j_{1}})\cap env_{I}(a_{j_{2}})$.
  
  Suppose that $b$ is separated from $I$ by cycles $d_{1}$ and $d_{2}$, where $d_{i}$ is connected to some $c_{i}\in\S(I)_{j_{i}}$ by some path $p_{i}$ that does not pass through $nod(I)$.
  
  From Observation \ref{cycles} applied to $d_{1},d_{2},a_{2},b$ with $b$ in place of $a$ and $w$ it follows that $\ppb{d_{1}}{b}\cup d_{1}\subseteq\ppb{d_{2}}{b}\cup d_{2}$. The symmetric argument yields the opposite inclusion so that $d_{1}=d_{2}$, contradicting the defining property of $nod(I)$.
  
  
  Similarly, if $d_{2}$ as above exists but there is also $d_{1}\in\S(I)_{j_{1}}$ such that $b\in\ppb{d_{1}}{a_{j_{1}}}\cup d_{1}$, then an analogous application of \ref{cycles} yields $a_{j_{1}}\in\ppb{d_{1}}{a_{j_{1}}}\subseteq\ppb{d_{2}}{b}$, contradicting the fact that $d_{2}$ separates $b$ and $I$.
  
  If $b\in\ppb{d_{1}}{a_{j_{1}}}\cap\ppb{d_{2}}{a_{j_{2}}}$ where $d_{i}\in\S(I)_{j_{i}}$, then $d_{1}$ and $d_{2}$ have to cross, contradicting Lemma \ref{one point intersection}. The remaining cases are straightforward from the defining property of $nod(I)$.
  \end{proof}
  
  The following can be easily checked by inspecting the definitions above:
  \begin{lemma}
  Given $b\in env_{I}(a_{j})$ there exists some formula $\phi(x,y;z)$ such that $\models\phi(a_{j},a_{j'},b)$ for any $j'\in J\setminus\{j\}$ and $\phi(a'_{j_{0}},a'_{j_{1}},z)$ implies $z\in env_{I}(a'_{j_{0}})$ for any free indiscernible sequence $(a'_{j})_{j\in J'}$.
  \end{lemma}

  \begin{corollary}
  \label{free envelopes and indiscernibility} For $j\in J$ the sequence $I$ is not indiscernible over any point $b\in env_{I}(a_{j})$.
  \end{corollary}
  \begin{proof}
  Fix $b\in env_{I}(a_{j_{0}})$, $j_{0}\in J$ and consider the formula $\phi(x,y,z)$ as above. Pick $j_{0}<j_{1}<j_{2}$ in $J$. Then $\phi(a_{j_{0}},a_{j_{2}},b)$ and if $I$ is indiscernible over $b$ we then have  $\phi(a_{j_{1}},a_{j_{2}},b)$, so that $b\in env_{I}(j_{1})$, contradicting \ref{intersection at the node}.
  \end{proof}
  
  \begin{corollary}
  \label{invariance free envelope} Any automorphism of $\mathbb{G}^{\iota}$ preserving $I$ setwise must preserve $env(I)$ setwise.
  \end{corollary}
  
  \begin{lemma}
  \label{free not indiscernible}Let $c$ be a cycle. Then either
  $env(I)$ is contained on one side of $c$ (so in particular $env(I)\cap c=\emptyset$) or  $c\subseteq env_{I}(a_{j})\cup nod(I)$ for some $j\in J$.
  In the latter case we have $b\sim_{c} b'$ for all $b,b'\in (env_{I}(a_{j})\cup nod(I))^{c}$.
  \end{lemma}
  \begin{proof}
  If $a_{j}\in c$ for some $j$, then $c\nsubseteq env_{I}(a_{j})$, so assume not.
  Likewise, if $c$ separates $a_{j_{0}}$ from $I\setminus\{a_{j_{0}}\}$, then $c\subseteq\S(I)_{j_{0}}$. So assume $c$ does not separate $I$. If it is connected to $\S(I)_{j}$ by a path that does not go through $nod(I)$ for some $j$, then $c\subseteq env_{I}(a_{j})$ by definition, so assume not.
  
  Then $\S(I)_{j}$ lies on the $I$-side of $c$ for all $j$ and therefore so does $\mathcal{H}(I)_{j}$ from definition \ref{free envelope}. Then for any $d\subseteq\mathcal{H}(I)_{j}\cup nod(I)$ that does not separate $I$ either $c\subseteq\nnb{d}{I}$ or $(\nnb{d}{I}\cup d)\cap\nnb{c}{I}=\emptyset$.
  It follows that all of $env(I)$ lies on the same side of $c$.
  If $c\subseteq env_{I}(a_{j})\cup nod(I)$ it is easy to see that one side of $c$ is contained in $env(I)(a_{j})$ (there are two possibilities according to whether $c$ separates $I$).
  \end{proof}
  
\section{Proof of dp-minimality}

  \begin{definition}
  For a large indiscernible sequence of elements $I=(a_{j})_{j\in J}$ that fits into a cage or nest $\D$
  we let $env(I)=int(\D)$. Notice that this does not depend on the choice of $\D$ by Corollary \ref{fitting}.
  
  Given an indiscernible sequence of tuples $(a_{j})_{j\in J}$, $a_{j}=(a_{j}^{l})_{l=1}^{n}$ such that
  $(a^{l}_{j})_{j\in J}$ is not constant for any $1\leq l\leq m$ we let $env(I)=\bigcup_{l=1}^{m}env((a^{l}_{j})_{j\in J})$.
  \end{definition}
  
  \begin{definition}
  Given a large indiscernible sequence of elements we write $\widebar{env}(I)=env(I)\cup pol(\mathcal{D})$ if $env(I)$ is the interior of a cage or nest $\mathcal{D}$ and $\widebar{env}(I)=env(I)\cup nod(I)$ if $I$ is free.
  \end{definition}

  \begin{lemma}
  \label{criterion indiscernibility}Let $I=(a_{j})_{j\in J}$, $a_{j}=(a_{j}^{l})_{l=1}^{n}$ be a large indiscernible sequence of tuples. Then $I$ is indiscernible over a set $B$ of parameters if and only if $B\cap env(I)=\emptyset$.
  \end{lemma}
  \begin{proof}
  We have already seen that if $B\cap env(I)\neq\emptyset$, then $I$ is not indiscernible over $B$ (Lemma \ref{parameters and cages} and Corollary \ref{free envelopes and indiscernibility}).
  
  Let $\lambda:J\to J$ an order-preserving bijection. Take an automorphism $\sigma$ of $\G^{\iota}$ sending
  $a_{j}$ to $a_{\lambda(j)}$. It is clear from the definition that $\sigma$ preserves  $env(I^{l})$ for each of the non-constant components of $I^{l}$ of $I$. We claim that $\sigma_{*}:=\sigma_{\restriction env(I)}\cup id_{env(I)^{c}}$ is an automorphism of $\G^{\iota}$.
  At the level of graphs this is clear from Lemma \ref{pole crossing} and Observation \ref{not through node} as long as $\sigma$ does not flip the two poles of the cage associated to $I^{l}$ for some $l$. However, we can always choose $\sigma$ so that this is the case, since the group of order preserving automorphisms of $J$ does not contain a subgroup of index two.
  
  Given a cycle $c$ we need to show that $\sim_{\sigma^{*}(c)}$ is the psush-forward of $\sim_{c}$ by automorphism $\sigma^{*}$. Fix $b,b'$ outside of $c\cup env(I)\cup c$ and $b''\in env(I)\cup c$.
  It suffices to show that $b\sim_{c}b'$ if and only if $b\sim_{\sigma^{*}(c)}\sigma(b)$.
  t
  
  Assume first that $c\subseteq \widebar{env}(I^{l})$ for some $1\leq l\leq n$ and $c$ does not separate $I^{l}$ into two subsequences unbounded on one side (this happens, in particular, if $c\subseteq \widebar{env}(I^{l})$ and $I^{l}$ is free).
  Then any two elements outside of $\widebar{env}(I^{l})$ are on the same side of $c$  (see Lemma \ref{free not indiscernible}) and since $\sigma(c)=\sigma^{*}(c)$ has the same property we have $\sigma(b)\sim_{\sigma(c)}b$.
  Then $b\sim_{c}b'$ is equivalent to $ \sigma(b)\sim_{\sigma(c)}\sigma(b')$, in turn equivalent $b\sim_{\sigma(c)}\sigma(b')$ and we are done.
  
  By virtue of Corollaries \ref{cycles and cages} and \ref{free not indiscernible} and Lemma \ref{cycles and cages} the other possibility is that $c$ contains a collection of finitely many subpaths $p_{1},\dots p_{r}$ with disjoint interior, where for each $1\leq i\leq m$ either $p_{i}$ crosses from pole to pole a cage $\mathcal{D}^{i}=(\delta^{i}_{l})_{l\in L}$ in which $I^{k_{i}}$ fits or $r=1$ and $p_{1}=c$ is contained in a nest $\mathcal{D}^{1}$ in which some $I^{k_{1}}$ fits and separates $I^{k_{1}}$ into two subsequences unbounded on one side.
  
  We might as well just consider the first subcase, since the other is almost formally identical. It is easy to see that $b'\sim_{\sigma^{*}(c)}b''$ if and only if $b'\sim_{\sigma(c)}b''$ for any $b,b''\in int(\D^{l})$.
  If $d\in int(\D^{i})\setminus D^{i}_{l,l'}$ for some $l<l'$ in $L$ such that
  $c\cap int(\D^{i})\subseteq D^{i}_{l,l'}$. Then it is easy to see that the truth value of $b\sim_{c'}d$ is the same for any $c'$ that is the result of replacing $p_{i}$ by $\delta^{i}_{l''}$ for some $l''\in (l,l'')$ and of replacing $p_{i'}$ by any path between the poles of $\C^{i'}$ for $1\leq i'\leq m$, $i\neq i'$. The same holds for $\sigma^{*}(c)$.
  It follows that if $j$ is large enough
  $a^{k_{i}}_{j}\sim_{c}b$ if and only if $a^{k_{i}}_{j}\sim_{\sigma^{*}(c)}b$ while on the other hand $a^{k_{i}}_{j}\sim_{c}b'$ if  and only if $ \sigma(a^{k_{i}}_{j}) \sim_{\sigma(c)}b'$ if and only if $a^{k_{i}}_{j}\sim_{\sigma^{*}(c)}\sigma(b')$. This concludes the proof.

  %
  
  \end{proof}
  
  \begin{lemma}
  \label{mutually indiscernible}Two large indiscernible sequences of tuples $I =(a_{j})_{j\in J}$, $I'=(a'_{j})_{j\in J'}$ are mutually indiscernible only if $env(I)\cap env(I')=\emptyset$.
  \end{lemma}
  \begin{proof}
  We may assume $I$ and $I'$ are sequences of elements.
  By the previous Lemma we know that $env(I )\cap I'=env(I')\cap I=\emptyset$.
  
  There are three possible cases.
  \subsection*{Both $I$ and $I'$ fit into a cage or nest}
    Assume $I$ fits into $\mathcal{D}$ and $I'$ into $\mathcal{D}'$, where each of $\mathcal{D}=(\delta_{l})_{l\in L}$ and $\mathcal{D}'=(\delta_{l})_{l\in L'}$ is either a large cage or a large nest.
    
    If $pol(\D)\neq pol(\D')$ it follows from Lemma \ref{cage compatibility} that either $int(\D)$ and $ int(\D')$ are disjoint or one of the two sets is contained in the other, contradicting the fact that $I\cap int(\D')=I'\cap int(\D)=\emptyset$.
    
    Suppose now $pol(\D)=pol(\D')$. Consider first the case in which $\D$ and $\D'$ are nests.
    For any $l\in L$ either $\delta_{l}\subset ext(\D')$ or $\delta_{l}\subseteq D_{l'_{0},l'_{1}}$ for a pair $l'_{0},l'_{1}\in L'$.
    
    In the latter case it cannot be that $\delta'_{l'_{0}}$ and $\delta'_{l'_{1}}$ lie on the same side of $\delta_{l}$ for then one side of $\delta_{l}$ must be contained in $D'_{l'_{0},l'_{1}}$ and thus
    $I\cap int(\D')\neq \emptyset$. So in that case $\delta_{l}$ must separate the sequence $(\delta'_{l'})_{l'\in L}$ into two unbounded intervals.
    
    If $\delta_{l}$ lies in $ext(\D')$ for all $l\in L$, then either $int(\D')\subseteq int(\D)$, which is impossible, or $ int(\D)\cap int(\D')$.
    We are left with two possibilities:
    \begin{itemize}
    \item There are $l_{0}<l_{1}<l_{2}\in L$ such that $\delta_{l_{k}}$ lies in $\ci(\D')$ for $k=0,1,2$. In this case $D_{l_{0},l_{2}}\subseteq int(\D')$ and $I\cap int(\D')\neq\emptyset$.
    \item There are $l_{0},l_{1}\in L$ with $\delta_{l_{0}}\subseteq \ci(\D')$ and $\delta_{l_{1}}\subseteq ext(\D')$. It follows that $D_{l_{0},l_{1}}$ contains $\dd'_{l'}$ for infinitely many values of $l'\in L$ and thus $D_{l_{0},l_{1}}\cap I'\neq\emptyset$.
    \end{itemize}
    
    The case analysis when $\D$ and $\D'$ is entirely analogous.
    %
    
  \subsection*{The sequence $I$ fits into a cage or nest and $I'$ is free }
    
    Suppose that $I$ fits into a large cage or nest $\mathcal{D}=(\delta_{l})_{l\in L}$. Let $c$ be a cycle. By Lemmas \ref{cycles and cages} and \ref{cycles and nests} we have three possibilities for $c$:
    \begin{enumerate}
    \item \label{q1}$c\subseteq ext(\D)$
    \item \label{q2}$c\subseteq D_{l_{1},l_{2}}\cup\delta_{l_{1},l_{2}}$ for $l_{1}<l_{2}$ in $L$.
    \item \label{q3}$\mathcal{D}=(\delta_{l})_{l\in L}$ is a cage and $c$ is of the form $p*q$, where $p$ lies in $ext(\D)$ and $q$ crosses $int(\D)$ from pole to pole
    \end{enumerate}
    
    If $c\in \S(I')_{j'}$ for some $j'\in J'$, then we can rule out \ref{q2}
    since it implies that either $I\setminus D_{l_{1},l_{2}}\subseteq\X(I')_{j'}$ or $\X(I')_{j'}\subseteq int(\D)$. Likewise, we can rule out \ref{q3}, since it implies $I\cap\X(I')_{j'}\neq\emptyset$. Since $I\cap env(I')=\emptyset$, we conclude that $\X(I')_{j'}\subseteq ext(\mathcal{D})$.
    
    Observe that any path $q$ which starts at $\X(I')_{j'}$ and does not pass through $nod(I')$ must stay in $ext(\mathcal{D})$. Otherwise $q$ crosses $pol(\mathcal{D})$ and indiscernibility of $I'$ over $I$ implies any two such sets can be joined by a path avoiding $nod(I')$, against Lemma \ref{intersection at the node}.
    
    Suppose now $c$ is a cycle that is connected to $\S(I')_{j'}$ by a path $q$ that does not go through $nod(I')$. The preceding remarks exclude cases \ref{q2} and \ref{q3} above, so $c\subseteq ext(\D)$. In case \ref{q1}. Recall that $\nnb{c}{I'}\subseteq env_{I}(a'_{j'})$. Since $I\cap env_{I}(a'_{j'})=\emptyset$ we must have $\nnb{c}{I'}\subseteq ext(\mathcal{D})$ and thus $env(I')\subseteq ext(\mathcal{D})=env(I)^{c}$ as desired.
    
  \subsection*{Both $I$ and $I'$ are free}
    Suppose now both $I$ and $I'$ are free and $env_{I}(a_{j})\cap env_{I'}(a'_{j'})\neq\emptyset$ for some and hence for all $j\in J$ and $j'\in J'$.
    Fix $j_{0}\in J$ such that $env_{I}(a_{j_{0}})$ does not intersect $nod(I')$.
    
    Recall that $\mathcal{H}(I)_{j}$ is the collection of vertices in $env_{I}(a_{j_{0}})$ that can be joined to $\X(I)_{j}$ by a path that does not pass through $nod(I)$. As a consequence of Lemma \ref{external are connected}, the subset $K_{j_{0}}$ of those vertices in $\mathcal{H}(I)_{j}$ that are not contained in $\ppb{c}{a_{j}}$ for a cycle $c\in\S(I)_{j}$ is connected.
    
    It follows that $env_{I'}(a'_{j'})$ cannot intersect $K_{j_{0}}$ for any $j'\in J'$, since otherwise by indiscernibility of $I'$ over $I$ the same would hold for any other $j''\in J'$, and there would exist a path from $env_{I'}(a'_{j'})$ to $env_{I'}(a'_{j''})$ avoiding $nod(I')$, contrary to \ref{intersection at the node}.
    
    Consider any $c'\in\S(I')_{j'}$. We claim that $c'\cap\X(I)_{j_{0}}=\emptyset$.
    Suppose not. If $c'$ contains points both in and outside $\X(I)_{j_{0}}$, then it has to intersect $K_{j_{0}}$. So in fact $c'\subseteq\X(I)_{j_{0}}$. But then it follows from Lemma \ref{merging cycles} that $c'\subset\ppb{c}{a_{j_{0}}}\cup c$ for some $c\in S(I)_{j_{0}}$, which implies that $I'\cap env(a_{j_{0}})\neq\emptyset$, since $c'$ separates $a_{j'}$ and $I'\setminus\{a_{j'}\}$.
    
    We conclude that $\X(I')_{j_{'}}\cap\X(I)_{j_{0}}=\emptyset$. Consider now $c\subset\H(I)_{j_{0}}\cup nod(I)$, $c\nin\S(I)_{j}$. We claim that $c'\subseteq\nnb{c}{I}$. By Lemma \ref{merging cycles} and the discussion above, we may assume that $c$ contains point in $K_{j_{0}}$, so $c\cap c'=\emptyset$. If $c'\subseteq\nnb{c}{I}$, then either $I\subseteq\ppb{c'}{a_{j'}}$ or $\ppb{c'}{a_{j'}}\subseteq\nnb{c}{I}\subseteq env_{I}(a_{j_{0}})$: a contradiction either way. So $\ppb{c'}{a'_{j'}}\cap\nnb{c}{I}=\emptyset$. We conclude that $env_{I}(a_{j_{0}})\cap\X(I')_{j'}=\emptyset$.
    
    Since $K_{j_{0}}\cap env_{I'}(a_{j'})=\emptyset$, it follows that $\H(I')_{j'}\cap env_{I}(a_{j_{0}})=\emptyset$. Let $c'\subset\H(I')_{j'}$ be a cycle that does not separate $I'$. It easily follows that either $env_{I}(a_{j_{0}})\subseteq\nnb{c'}{I'}\subseteq env_{I'}(a_{j'})$ or
    $env_{I}(a_{j_{0}})\cap\nnb{c'}{I'}=\emptyset $: a contradiction.

    \end{proof}

    Given the tools above the proof of dp-minimality is straightforward.
    
    \begin{proof}[Proof of Theorem \ref{main theorem}]
    Suppose we are given two sequences of tuples  $I_{1}$ and $I_{2}$ that are mutually indiscernible. We may assume both sequences are large. Lemma \ref{mutually indiscernible} implies $env(I_{1})\cap env(I_{2})=\emptyset$ so $b\nin env(I_{i})$ for some $i\in\{1,2\}$ so that $I_{i}$ is indiscernible over $b$ by \ref{criterion indiscernibility}.
    \end{proof}
    
    \bibliographystyle{plain}
    \bibliography{bibliography}

\end{document}